\documentclass[11pt,twoside]{article}
\usepackage{fancyhdr}
\usepackage[colorlinks,citecolor=blue,urlcolor=blue,linkcolor=blue,bookmarks=false,breaklinks]{hyperref}
\usepackage{amsfonts,epsfig,graphicx}
\usepackage{afterpage}
\usepackage{amsmath,amssymb,amsthm} 
\usepackage{fullpage}
\usepackage[T1]{fontenc} 
\usepackage{epsf} 
\usepackage{graphics} 
\usepackage{amsfonts,amsmath}
\usepackage[sort,authoryear]{natbib} 
\usepackage{psfrag,xspace}
\usepackage{color,etoolbox}
\usepackage{enumitem}

\usepackage{tikz}

\usepackage{mathtools}

\newcommand{\inprod}[2]{\ensuremath{\langle #1 , \, #2 \rangle}}

\newtheorem{theorem}{Theorem} 
 
\newtheorem{lemma}{Lemma}

\newtheorem{corollary}{Corollary}

\newcommand{\abscont}{\mathcal{P}_{\mathrm{ac}}(\Omega)}

\newcommand{\abscontmom}{\mathcal{P}_{2,\mathrm{ac}}(\Omega)}

\newcommand{\semidualpq}{\mathcal{S}_{P,Q}}
\newcommand{\semidualphatqhat}{\mathcal{S}_{\widehat{P},\widehat{Q}}}
\newcommand{\semidualphatq}{\mathcal{S}_{\widehat{P},Q}}
\newcommand{\semidualpqhat}{\mathcal{S}_{P,\widehat{Q}}}
\newcommand{\probspace}{\mathcal{P}_2(\Omega)}
\newcommand{\probspacesquared}{\mathcal{P}_2(\Omega \times \Omega)}

\newcommand{\probone}{P}
\newcommand{\probtwo}{Q}

\newcommand{\nn}{\text{nn}}

\usepackage{macros}

\begin{document}

\begin{center} {\LARGE{\bf{
Stability Bounds for Smooth Optimal Transport Maps \\ \vspace{.2cm} 
 and 
their Statistical Implications}}}
\\

\vspace*{.3in}

{\large{
\begin{tabular}{ccccc}
Sivaraman Balakrishnan$^{\dagger}$ and Tudor Manole$^{\diamond}$\\
\end{tabular}

\vspace*{.1in}

\begin{tabular}{ccc}
$^{\dagger}$\,Department of Statistics and Data Science \\
Machine Learning Department \\
\end{tabular}

\begin{tabular}{c}
Carnegie Mellon University \\
\texttt{siva@stat.cmu.edu}
\end{tabular}

\vspace{0.2in}

\begin{tabular}{ccc}
$^{\diamond}$\,Statistics and Data Science Center \\
Massachusetts Institute of Technology\\
\texttt{tmanole@mit.edu}
\end{tabular}

\vspace*{.2in}
}}

\vspace*{.2in}

\today
\vspace*{.2in}

\begin{abstract}
\noindent

\noindent We study estimators of the optimal transport (OT) map between two probability distributions. We focus on plugin estimators derived from the OT map between estimates of the underlying distributions. 
We develop novel stability bounds for OT maps which generalize those in past work, and allow us to reduce the problem of optimally estimating the transport map to that of optimally estimating densities in the Wasserstein distance. 
In contrast, past work provided a partial connection between these problems and relied on regularity theory for the Monge-Amp\`ere equation to bridge the gap, a step which required unnatural assumptions to obtain sharp guarantees. 
We also provide some new insights
into the connections between stability bounds which arise in the analysis of plugin estimators and 
growth bounds for the semi-dual functional which arise in the analysis of Brenier potential-based estimators of the transport map. 
We illustrate the applicability of our new stability bounds by revisiting the smooth setting
studied by \citet{manole2024plugin}, analyzing two of their estimators under more general conditions.
Critically, our bounds do not require smoothness or boundedness assumptions on the underlying measures. As an illustrative application, we develop and analyze a novel tuning parameter-free estimator for the OT map between two strongly log-concave distributions.
\end{abstract}
\end{center}

\section{Introduction}
The field of optimal transport (OT) is centered on the following question: Given two probability measures $\probone$ and $\probtwo$ supported in $\mathbb{R}^d$, how can we transport $\probone$ to $\probtwo$ while minimizing the transportation cost? For the squared Euclidean distance cost, an OT map $T_0$ from $\probone$ to $\probtwo$
is any solution to the {\it Monge problem}~\citep{monge1781},
\begin{equation}
\label{eqn:monge}
\argmin_{T:T_\# P = Q} \int \norm{x-T(x)}_2^2 dP(x),
\end{equation}
where the minimizer is chosen from transport maps between $P$ and $Q$, 
that is, the set of Borel-measurable functions $T:\bbR^d \to \bbR^d$
such that $T_{\#} P := P(T^{-1}(\cdot)) = Q$. Intuitively, a transport map $T$ ``transports'' a random variable $X$ with distribution $P$ to a random variable $T(X)$ which has distribution $Q$.

One of the central preoccupations of the field of \emph{statistical} optimal transport \citep{chewi2024} is to understand
the estimation of objects which arise from the OT framework\textemdash for instance, the optimal transport map $T_0$ or
the optimal transport cost\textemdash when the two measures are unknown, but we have samples from them.
Our focus in this paper is on estimating the OT map from samples. Concretely, given i.i.d. samples $X_1,\ldots, X_n \sim \probone$ and $Y_1,\ldots, Y_m \sim \probtwo$ our goal is to construct an estimate $\widehat{T}_{nm}$ of the OT map $T_0$ between $\probone$ and $\probtwo$, for which the~risk 
\begin{align}
\label{eqn:risk}
R(\widehat{T}_{nm}, T_0) := \mathbb{E} \left[\int \|\widehat{T}_{nm}(x) - T_0(x)\|_2^2 d\probone(x)\right]
\end{align}
is small. Here, the outer expectation is taken over the randomness of the two samples. Two of the broad classes of estimators that have been studied in the literature are \emph{plugin estimators}, which construct estimates of~$T_0$ based on the optimal transport 
map between estimates~$\widehat{\probone}$ and~$\widehat{\probtwo}$ of the underlying distribution \citep{chernozhukov2017,gunsilius2021,ghosal2022,deb2021a,manole2024plugin}, and \emph{dual estimators}, which are based on a characterization (see Theorems~\ref{thm:brenier} and~\ref{thm:santambrogio}) of the optimal transport map as the gradient of a convex function $\varphi_0$ which solves the so-called \emph{semi-dual} optimization problem:
    \begin{align}
    \label{eqn:semidual}
        \varphi_0 \in  \argmin_{\varphi \in L^1(P)} \int \varphi d\probone + \int \varphi^*dQ.
    \end{align}
We refer to the quantity, 
\begin{align}
    \label{eqn:semi-dual-functional}
    \mathcal{S}_{P,Q}(\varphi) := \int \varphi d\probone + \int \varphi^* dQ
\end{align}
as the \emph{semi-dual functional}.
Dual estimators approximate the solution to the program~\eqref{eqn:semidual} building on standard ideas from the M-estimation literature \citep{vandervaart1996,vandegeer2000}, and then compute an estimate $\widehat{T} = \nabla \widehat{\varphi}$ \citep{hutter2021,divol2022a,ding2024}.
In contrast to plugin estimators, the computation
of dual estimators is often impractical because computing the convex conjugate $\varphi^*$ in the program above can be challenging. On the other hand, one can directly bring to bear powerful ideas from the analysis of M-estimation to sharply characterize the statistical performance of dual estimators and as a result they serve as a useful information-theoretic
benchmark.

Central to the study of statistical properties of plugin estimators are \emph{stability bounds} which aim to quantify the distance between the optimal transport map 
between $\widehat{\probone}$ and $\widehat{\probtwo}$ and the one between $\probone$ and $\probtwo$, i.e. they quantify the \emph{stability} of the optimal transport map between two distributions to perturbations of the distributions~\citep{deb2021,manole2024plugin}. The study of dual estimators has on the other hand mirrored the study of $M$-estimators. As highlighted by the work of \cite{hutter2021}, a key quantity in analyzing dual estimators is the so-called \emph{growth rate} of the semi-dual functional, i.e. the rate at which the difference $\mathcal{S}_{P,Q}(\varphi) - \mathcal{S}_{P,Q}(\varphi_0)$ grows as a function of the distance between $\varphi$ and $\varphi_0$. Concretely, \citet{hutter2021} showed that under a certain curvature assumption on~$\varphi_0$, the growth rate is quadratic in the following sense:
\begin{align}
\label{eqn:dualstability}
    \mathcal{S}_{P,Q}(\varphi) - \mathcal{S}_{P,Q}(\varphi_0) \asymp \int \|\nabla \varphi(x) - \nabla \varphi_0(x)\|_2^2 d\probone(x).
\end{align}
This quadratic growth, together with empirical process techniques from the analysis of M-estimators, has led to sharp statistical bounds for dual estimators \citep{hutter2021,divol2022a}.
In our work, we shed some light on the connections between the growth rate of the semi-dual functional and stability bounds by showing that many existing and new stability bounds 
can be derived by studying the growth rate of appropriate semi-dual functionals.

\vspace{.3cm}

\noindent {\bf Our Contributions: } Our primary contribution (Theorem~\ref{thm:stability}) is an
improvement to the stability bounds of \citet{manole2024plugin}. We recover
their one-sample stability bound and improved versions of their two-sample and empirical stability 
bounds as direct consequences of a single 
unified result. This new stability bound 
has immediate implications for the analysis of plugin estimators of the OT map. As an important example,
we obtain new results for the estimation of the optimal transport map between distributions with smooth densities, in Section~\ref{sec:smooth}. Similar results were previously obtained
by \citet{manole2024plugin} \emph{conditional} on a uniform boundary regularity assumption holding. In order to 
obtain unconditional results, the authors made the unnatural assumption that the sampling distributions were supported on the $d$-dimensional flat torus. Using our new stability bound we provide an unconditional result with a much simpler proof, completely bypassing the regularity issues. We provide another illustration of the generality of our stability bounds by using them to provide guarantees for a new, practical estimator of the optimal transport map between two log-smooth and strongly log-concave distributions.

\vspace{.5cm}

\noindent {\bf Organization: } We begin with some background on optimal transport and introduce the semi-dual functional
formally in Section~\ref{sec:background}. We state our new stability bound in Section~\ref{sec:curv} and discuss how it relates to primal and dual stability bounds developed in past work. In Section~\ref{sec:smooth} we 
develop consequences of our stability bound in estimating the optimal transport map between two smooth distributions, highlighting how we are able to close some important gaps in this literature. In Section~\ref{sec:logconcave} we consider the problem of estimating the optimal
transport map between two strongly smooth and strongly log-concave distributions. We conclude in Section~\ref{sec:discussion}.

\vspace{.5cm}

\noindent {\bf Notation: } For a domain $\Omega \subseteq \mathbb{R}^d$,
we let $\calP(\Omega)$ denote the set
of Borel probability measures supported
on $\Omega$, and $\probspace$ be the subset of 
such probability measures with finite second moment. 
Similarly, let $\calP_{\mathrm{ac}}(\Omega)$
(resp.\,$\calP_{2,\mathrm{ac}}(\Omega)$)
denote the set of probability
measures in $\calP(\Omega)$ (resp.\,$\calP_2(\Omega)$)
which are absolutely continuous with respect
to the Lebesgue measure.
We denote by $\calC^s(\Omega)$
the standard H\"older space with
real and positive exponent $s > 0$
over~$\Omega$~\citep{gilbarg2001},
which is sometimes
denoted $\calC^{\lfloor s \rfloor,s-\lfloor s \rfloor}(\Omega)$.
Given a vector field $T:\bbR^d\to\bbR^d$
and a measure $\mu$ on 
$\bbR^d$, we write
$\|T\|_{L^2(\mu)}^2 = 
\int_{\bbR^d} \|T(x)\|_2^2 d\mu(x)$. 
For a pair of distributions $P,Q \in \calP(\bbR^d)$, we denote by $\varphi_{P,Q}$ a \emph{Brenier potential}, i.e. \emph{any} potential which solves the semi-dual problem~\eqref{eqn:semidual}. For a convex function $\varphi$ defined on $\Omega$, we denote by $\varphi^*$ its Fenchel conjugate. In general, $\varphi^*$ takes values in the extended reals and is defined by $\varphi^*(y) := \sup_{x \in \Omega} \left[\inprod{x}{y} - \varphi(x)\right].$ For sequences 
$(a_n)_{n=1}^\infty$ and $(b_n)_{n=1}^\infty$, we write $a_n \lesssim b_n$ if there exists
$C > 0$ such that $a_n \leq C b_n$ for all $n \geq 1$. Following the convention in nonparametric statistics, the constant $C$ above can depend on the dimension~$d$, and other problem parameters when explicitly mentioned, but is otherwise universal.

\section{Background}
\label{sec:background}

We work with the following statistical setup: We obtain i.i.d.~samples $X_1,\ldots,X_n \sim P$ and $Y_1,\ldots, Y_m \sim Q$, where $P,Q \in \probspace$ and the distributions $P,Q$ belong to some structured collection of distributions $\mathcal{R}$.  In Section~\ref{sec:smooth} our focus will be on the setting where $\mathcal{R}$ is a collection of 
distributions with H\"{o}lder smooth densities, while in Section~\ref{sec:logconcave} we will consider the case where $\mathcal{R}$ is the collection of strongly log-concave and log-smooth distributions. 
Our goal is to construct an estimator $\widehat{T}_{nm}$ such that the risk in~\eqref{eqn:risk} is uniformly small over the collection $\mathcal{R}$, i.e. we are interested in constructing estimators for which the maximum risk:
\begin{align*}
\sup_{P,Q \in \mathcal{R}} R(\widehat{T}_{nm}, T_0) 
\end{align*}
is small. 

The Monge problem in~\eqref{eqn:monge} does not always have a solution, and the push-forward constraint is nonlinear in general. These aspects motivate the study of the Kantorovich relaxation \citep{kantorovich1948,kantorovich1942}
of identifying an \emph{optimal coupling} between $P$ and $Q$, defined~as: 
\begin{align}
\label{eqn:kantorovich}
\pi_0 \in \argmin_{\pi \in \Pi(P,Q)} \int \|x - y\|_2^2 d\pi(x,y),
\end{align}
where the set $\Pi(P,Q)$ denotes the set of couplings of $P$ and $Q$, i.e.
\begin{align*}
\Pi(P,Q) = \{ \pi \in \probspacesquared: \pi( \cdot \times \Omega) = P, \pi( \Omega \times \cdot) = Q \}.
\end{align*}
In contrast to the Monge problem~\eqref{eqn:monge}, the Kantorovich problem~\eqref{eqn:kantorovich} is always feasible, and is a linear program. From~\eqref{eqn:kantorovich} we can derive the following intuitive optimality property: for \emph{any} pair
of random variables $U$ with distribution $P$, and $V$ with distribution $Q$ we have that:
\begin{align}
\label{eqn:maxcorr}
\mathbb{E}_{(X,Y) \sim \pi_0} \inprod{X}{Y}  \geq \mathbb{E} \inprod{U}{V},
\end{align}
i.e. the optimal coupling, among all pairs with given marginals, defines a pair of random variables with maximum correlation.

\subsection{Brenier's Theorem and Alternatives}
The Kantorovich problem is a relaxation of Monge's problem---every transport map between distributions defines a coupling of the distributions. In contrast to Monge's problem, the Kantorovich relaxation is always feasible, and is a linear program.
A cornerstone result in OT theory, is Brenier's theorem which relates the solutions of the Monge and Kantorovich problems:
\begin{theorem}[Brenier's Theorem]
\label{thm:brenier}
Let $P \in \abscontmom$ and $Q \in \probspace$. 
Then, the following assertions hold.
\begin{enumerate}
\item  
There exists an optimal transport map $T_0$ 
pushing $P$ forward onto $Q$
which takes the form $T_0 = \nabla\varphi_0$ for a convex function $\varphi_0: \mathbb{R}^d \to \mathbb{R}$
which solves the semi-dual problem~\eqref{eqn:semidual}. Furthermore, $T_0$
  is uniquely determined $P$-almost everywhere.  
\item If we further have $Q \in \abscontmom$, then $S_0 := \nabla \varphi_0^*$ is 
the $Q$-almost everywhere uniquely
determined optimal transport map pushing $Q$ forward onto $P$. Furthermore, for Lebesgue-almost every $x, y \in \Omega$, 
$$\nabla\varphi_0^* \circ \nabla \varphi_0(x) = x, \quad \nabla\varphi_0\circ\nabla\varphi_0^*(y) = y.$$
\end{enumerate}
\end{theorem}
\noindent Our stability bounds
also hold in some cases where Brenier's theorem does not  apply. Indeed, we will generally
only require the measures $P$ and $Q$ to be chosen
such that the following assumption holds:
\begin{enumerate}[leftmargin=1.65cm,listparindent=-\leftmargin,label=\textbf{A0}]
\item  \label{ass:brenier_regularity} There exists a strictly convex and differentiable function $\varphi_0:\Omega\to\bbR$, such that $\varphi_0 \vert_{\text{supp}(P)}$ solves~\eqref{eqn:semidual}, and $T_0:=\nabla\varphi_0$ is a vector field from $\Omega$ into itself.
\end{enumerate}
\noindent 
With a slight abuse of terminology, 
we will refer to the map
$\varphi_0:\Omega\to\bbR$
in condition~\ref{ass:brenier_regularity}
as a Brenier potential.
By Brenier's theorem, this condition is  
satisfied whenever $P\in \calP_{2,\mathrm{ac}}(\Omega)$.
However, it is also satisfied 
when $P$ is an arbitrary measure, 
provided that $Q$ is the pushforward of $P$
under a map $\nabla\varphi_0:\Omega\to\Omega$,
with $\varphi_0$ a strictly convex and differentiable
function.
The following result (which can be deduced from Proposition 1.15 of~\cite{santambrogio2015})
characterizes the optimal
transport problem between $P$ and $Q$ when
assumption~\ref{ass:brenier_regularity} holds:
\begin{theorem}
\label{thm:santambrogio}
Let $P,Q \in \probspace$, and suppose that~\ref{ass:brenier_regularity} holds.
Then, there exists a unique optimal transport
coupling between $P$ and $Q$, 
which is induced by the optimal transport maps~$T_0 := \nabla\varphi_0$ between $P$ and $Q$, 
and $S_0 := \nabla \varphi_0^*$ between $Q$ and $P$.
\end{theorem}
\noindent Theorems~\ref{thm:brenier} and~\ref{thm:santambrogio} provide sufficient conditions for the existence of a unique OT map  between $P$ and $Q$, and connects this unique map
to the solution of the semi-dual program~\eqref{eqn:semidual}. These results are important as they allow us to study Monge's problem via its linear programming relaxation. This in turn allows us to bring tools from convex analysis and convex duality to bear on the OT problem.
The dual perspective will play a key role in the development of stability bounds. 

\subsection{Wasserstein Distance and the Semi-Dual Functional}
The optimal value of the Kantorovich problem in~\eqref{eqn:kantorovich} defines a metric on distributions known as the 2-Wasserstein distance:
\begin{align*}
W_2^2(P,Q) := \inf_{\pi \in \Pi(P,Q)} \int \|x - y\|_2^2 d\pi(x,y).
\end{align*}
The Wasserstein distance has a dual characterization via the identity:
\begin{align*}
W_2^2(\probone, \probtwo) = \int \| \cdot\|_2^2 dP + \int \| \cdot\|_2^2 dQ - 2 \inf_{\varphi \in L^1(P)} \semidualpq(\varphi),
\end{align*}
where the semi-dual functional $\semidualpq$ is defined in~\eqref{eqn:semi-dual-functional}. The semi-dual optimization problem~\eqref{eqn:semidual} can be seen as a dual optimization problem to Kantorovich's linear program~\citep{villani2003}. 

The semi-dual functional is central in our stability bounds.
Of particular interest to us will be a characterization of the semi-dual growth function as an integrated Bregman divergence between transport maps (see Appendix~\ref{app:bregmanproof})~\citep{muzellec2021,vacher2021convex}:
\begin{lemma}
    \label{lem:bregman}
    Suppose that $P,Q,\widetilde{P}, \widetilde{Q} \in \probspace$ and that~\ref{ass:brenier_regularity} holds. 
    Let $T_0 := \nabla \varphi_0$, let $\widetilde{\pi}$ denote an optimal coupling of $\widetilde{P}$ and $\widetilde{Q}$ and let $\widetilde{\varphi}$ denote a Brenier potential between $\widetilde{P}$ and $\widetilde{Q}$. Then:
    \begin{align}
    \label{eqn:bregman}
        \mathcal{S}_{\widetilde{P},\widetilde{Q}}(\varphi_0) - \mathcal{S}_{\widetilde{P},\widetilde{Q}}(\widetilde{\varphi}) = \int \left[ \varphi^*_{0}(y) - 
\varphi^*_0(T_0(x)) - \inprod{\nabla \varphi_0^*(T_0(x))}{y - T_0(x)}\right] d \widetilde{\pi}(x,y). 
    \end{align}
\end{lemma}
\noindent In cases when the coupling $\widetilde{\pi}$ is realized by a transport map $\widetilde{T}$ the expression~\eqref{eqn:bregman} simplifies to:
\begin{align*}
    \mathcal{S}_{\widetilde{P},\widetilde{Q}}(\varphi_0) - \mathcal{S}_{\widetilde{P},\widetilde{Q}}(\widetilde{\varphi}) = \int \left[ \varphi^*_{0}(\widetilde{T}(x)) - 
\varphi^*_0(T_0(x)) - \inprod{\nabla \varphi_0^*(T_0(x))}{\widetilde{T}(x) - T_0(x)}\right] d\widetilde{P}(x).
\end{align*}
Intuitively, this result shows that the growth of the semi-dual functional away from its minimizer effectively measures a ``distance'' between two transport maps, which in turn gives further insight into the crucial role played by the semi-dual functional in the estimation of transport maps.

\subsection{Canonical Estimators of the OT Map}
Having setup the necessary background on optimal transport we now describe the two classes of estimators that we study:
\begin{enumerate}
\item {\bf Plugin Estimators:} We consider two types of plugin estimators. \emph{Smooth estimators} of the OT map are obtained by first constructing estimates $\widehat{\probone}_n$ and $\widehat{\probtwo}_m$ of the measures $\probone$ and $\probtwo$, and then defining $\widehat{T}_{nm}$ to be the optimal transport map between $\widehat{\probone}_n$ and $\widehat{\probtwo}_m$. For $\hat T_{nm}$ to be well-defined, the estimators
$\hat P_n$ and $\hat Q_m$ need to be {\it proper}, in the sense that they define probability measures
in their own right, and are related by an optimal transport map (which is, for instance, the case if $\widehat P_n$ is absolutely continuous).

    \emph{Empirical estimators} are similar in spirit,  but we first compute an optimal transport coupling between the empirical measures:
\begin{equation} 
\label{eq:empirical_measures}
P_n = \frac 1 n \sum_{i=1}^n \delta_{X_i}, \quad \text{and}\quad Q_m = \frac 1 m \sum_{j=1}^m \delta_{Y_j}.
\end{equation}
    This empirical coupling, which is only defined at the sample points $X_1,\ldots,X_n$, can then be extended to the entire domain using ideas from nonparametric regression \citep{manole2024plugin},
    or using the so-called
    barycentric projection~\citep{deb2021}.
    We refer to empirical and smooth estimators as \emph{plugin} estimators.
    
    \item {\bf Dual Estimators:} An alternative class of estimators is based on the fact that the optimal transport map is the gradient of a convex function~$\varphi_0$ which solves the following \emph{semi-dual} optimization problem:
    \begin{align}
        \varphi_0 := \argmin_{\varphi \in L^1(P)} \int \varphi d\probone + \int \varphi^*dQ,
    \end{align}
    where $\varphi^*$ denotes the convex conjugate of $\varphi$. Given samples from the measures~$\probone$ and~$\probtwo$ a natural class of estimators for $\varphi_0$, is:
    \begin{align}
    \label{eqn:dual}
        \widehat{\varphi} = \argmin_{\varphi \in \Phi} \int \varphi dP_n + \int \varphi^* dQ_m,
    \end{align}
    where $\Phi$ is an appropriate function class.
    Under some regularity conditions, solving the empirical semi-dual optimization problem yields estimators for the transport map $\widehat{T}_{nm} = \nabla \hat{\varphi}$, which we refer to as dual estimators. 
    Dual estimators have been proposed and analyzed in many prior works \citep{hutter2021,divol2022a,vacher2021convex,ding2024}.
\end{enumerate}

\section{Stability Bounds}
\label{sec:curv}

We now present our main stability bound and discuss its consequences. In the analysis of plugin and dual transport map estimators the key technical condition, which endows the semi-dual optimization problem~\eqref{eqn:semidual} with a desirable growth property (see~\eqref{eqn:dualstability}), is that there exists a
Brenier potential $\varphi_0:\Omega\to\bbR$ from $P$ to $Q$ in the sense of condition~\ref{ass:brenier_regularity},
which is additionally smooth and strongly convex:
\begin{enumerate}[leftmargin=1.65cm,listparindent=-\leftmargin,label=\textbf{A1($\alpha$)}]
\item  \label{ass:strong-convexity} The function
$\varphi_0$ is  convex, continuously differentiable
over $\Omega$, and satisfies 
\begin{align*}
    \varphi_0(y) \geq \varphi_0(x) + \inprod{\nabla \varphi_0(x)}{y - x} + \frac{\alpha}{2} \|y - x\|_2^2,\quad
    \text{for all } x,y \in \Omega.
\end{align*}
\end{enumerate}
\begin{enumerate}[leftmargin=1.65cm,listparindent=-\leftmargin,label=\textbf{A2($\beta$)}]
\item  \label{ass:smoothness} The function
$\varphi_0$ is  convex, continuously differentiable
over $\Omega$, and satisfies 
\begin{align*}
    \varphi_0(y) \leq \varphi_0(x) + \inprod{\nabla \varphi_0(x)}{y - x} + \frac{\beta}{2} \|y - x\|_2^2,
    \quad \text{ for all } x,y\in\Omega.
\end{align*}
\end{enumerate}
\noindent A standard fact in convex analysis is that $\beta$-smoothness of $\varphi_0$ implies $1/\beta$-strong convexity of~$\varphi_0^*$ and $\alpha$-strong convexity of $\varphi_0$ implies $1/\alpha$-smoothness of $\varphi_0^*$ \citep{urruty1993}. When assumptions~\ref{ass:brenier_regularity},~\ref{ass:strong-convexity} and ~\ref{ass:smoothness}
hold then, by Theorem~\ref{thm:santambrogio},
there exists a $P$-almost everywhere uniquely defined transport map $T_0$ from $P$ to $Q$
which is bi-Lipschitz over $\Omega$, and whose
inverse is the $Q$-almost everywhere uniquely defined transport map from $Q$ to $P$.

Our main contribution in this section is the following. 
\begin{theorem}
\label{thm:stability}
Let $P, Q \in \probspace$, and assume~\ref{ass:brenier_regularity}, \ref{ass:strong-convexity} and~\ref{ass:smoothness} hold for some $\alpha, \beta > 0$. For any $\widehat{P}, \widehat{Q} \in \probspace$, let $\widehat{\pi}$ denote an optimal coupling of $\widehat{P}$ and $\widehat{Q}$. Then,
\begin{align}
\label{eqn:mainstability}
\frac{1}{\beta} \mathbb{E}_{(X,Y) \sim \widehat{\pi}} \| Y - T_0(X)\|_2^2 \leq 
\frac{1}{\alpha}  W_2^2(\widehat{Q},Q)  + \beta W_2^2(\widehat{P},P) + 2 W_2(\widehat{P},P) W_2(\widehat{Q},Q). 
\end{align} 
\end{theorem} 
\noindent At a high-level, this inequality provides an upper bound on a notion of discrepancy between an estimated optimal coupling $\widehat{\pi}$ and the target OT map $T_0$ in terms of the Wasserstein distances between $\widehat{P}$ and $P$, and $\widehat{Q}$ and $Q$. 
It is worth noting that this result holds even when the measures under consideration do not share the same support. This flexibility will be salient in Section~\ref{sec:logconcave}.  

Let us now highlight
several interesting consequences of this general stability bound. 

\vspace{.5cm}

\noindent {\bf One-Sample Stability Bounds:} An instructive special case to consider is when  $P \in \abscont$ is a known distribution, and we only obtain samples $Y_1,\ldots, Y_m \sim Q$. In this case, by Theorem~\ref{thm:brenier}, there is an 
optimal map $\widehat{T}$ between $P$ and $\widehat{Q}$, and our stability bound reduces to the following corollary:
\begin{corollary}
    \label{cor:onesample}
    Let $P \in \abscont, Q \in \probspace$, and assume~\ref{ass:brenier_regularity}, \ref{ass:strong-convexity} and~\ref{ass:smoothness} hold for some $\alpha, \beta > 0$. For any $\widehat{Q} \in \probspace$, 
let $\widehat{T}$ denote the OT map from $P$ to $\widehat{Q}$. Then,
\begin{align}
\label{eqn:onesamplestability}
\frac{1}{\beta} \mathbb{E}_{X \sim P} \| \widehat{T}(X) - T_0(X)\|_2^2 \leq \frac{1}{\alpha} W_2^2(\widehat{Q},Q). 
\end{align} 
\end{corollary}
\noindent This recovers a result of \citet{manole2024plugin} (see Theorem 6). We could also consider the case where the source measure is sampled, and the target measure is known:
\begin{corollary}
    \label{cor:onesample_target}
    Let $P, Q \in \probspace$, and assume~\ref{ass:brenier_regularity} and \ref{ass:smoothness} hold for some $\beta > 0$. For any $\widehat{P} \in \probspace$, 
let $\widehat{\pi}$ denote the optimal coupling between $\widehat{P}$ and $Q$. Then,
\begin{align}
\label{eqn:onesamplestability_target}
\frac{1}{\beta} \mathbb{E}_{(X,Y) \sim \widehat{\pi}} \| Y - T_0(X)\|_2^2 \leq \beta W_2^2(\widehat{P},P). 
\end{align} 
\end{corollary}
\noindent To our knowledge, this result is new, and surprisingly requires only smoothness of the Brenier potential $\varphi_0$ and not strong convexity. 

\vspace{.5cm}

\noindent {\bf Two-Sample Stability Bound: } In the general setting, the stability bound~\eqref{eqn:mainstability}
substantially strengthens the corresponding result in~\citet{manole2024plugin} (see Proposition 13).
In particular, under assumptions \ref{ass:strong-convexity} and~\ref{ass:smoothness}, their work was only able to establish an upper bound on the semi-dual growth of the form:
\begin{align*}
 \mathcal{S}_{\widehat{P},\widehat{Q}}(\varphi_0) - \calS_{\widehat{P},\widehat{Q}}(\varphi_{\widehat{P},\widehat{Q}}) \leq \left[\beta + \frac{1}{\alpha}\right] \left[W_2(\widehat{P},P) + W_2(\widehat{Q},Q)\right]^2,
\end{align*}
where~$\varphi_{\widehat{P}, \widehat{Q}}$ denotes
any Brenier potential in the OT problem from $\hat P$
to $\hat Q$.
The above display does not yield useful bounds on the risk of the plugin transport map estimate. To remedy this deficiency, 
\cite{manole2024plugin} instead needed to assume \emph{uniform bounds} on the smoothness and strong-convexity of the \emph{estimated} Brenier potential $\varphi_{\widehat{P},\widehat{Q}}$. This step weakened their results  considerably, requiring them to assume that the distributions under consideration were supported on the flat $d$-dimensional torus and that the target of inference was the torus OT map (the OT map with respect to a modified Euclidean metric). 
In contrast, building upon our improved stability bound we obtain sharp results for the usual plugin estimators of the OT map in the smooth setting, in Section~\ref{sec:smooth}.

\vspace{.5cm}

\noindent {\bf Empirical Stability Bound: } In the case when $\widehat{P}$ and $\widehat{Q}$ are empirical measures, we obtain the following corollary:
\begin{corollary}
\label{cor:empirical}
Let $P, Q \in \probspace$, and assume~\ref{ass:brenier_regularity}, \ref{ass:strong-convexity} and~\ref{ass:smoothness} hold for some $\alpha, \beta > 0$. Let $\widehat{\pi}$ denote an optimal coupling of the empirical distributions $P_n$ and $Q_m$. Then,
\begin{align*}
\frac{1}{\beta} \mathbb{E}_{(X,Y) \sim \widehat{\pi}} \| Y - T_0(X)\|_2^2 \leq \frac{1}{\alpha}  W_2^2(Q_m,Q)  + \beta W_2^2(P_n,P) + 2 W_2(P_n,P) W_2(Q_m,Q). 
\end{align*} 
\end{corollary}
\noindent This result strengthens Proposition 12 of~\citet{manole2024plugin}. 
To build some intuition, it is worthwhile to consider the situation where $n = m$, in which case the optimal coupling between $P_n$ and $Q_n$ is realized by an optimal transport map $\widehat{T}_n$.
Corollary~\ref{cor:empirical} then yields a bound on the \emph{empirical} error of the map $\widehat{T}_n$:
\begin{align*}
\frac{1}{n\beta} \sum_{i=1}^n \| \widehat{T}_n(X_i) - T_0(X_i)\|_2^2 \leq \frac{1}{\alpha}  W_2^2(Q_m,Q)  + \beta W_2^2(P_n,P) + 2 W_2(P_n,P) W_2(Q_m,Q). 
\end{align*} 
The convergence of the empirical measure in the Wasserstein distance has been extensively studied \citep{weed2019,fournier2015,lei2020}, and these results provide sharp bounds on the right-hand side of the above display. In Section~\ref{sec:logconcave} we illustrate how the map $\widehat{T}_n$ can be extended to the domain $\Omega$, providing guarantees on the $L^2(P)$ risk of the estimator $\widehat{T}_n$.

\vspace{.5cm}

\noindent {\bf Semi-Dual Quadratic Growth Bounds: } The stability bound~\eqref{eqn:mainstability} is obtained by upper and lower bounding the growth of a semi-dual functional. As we highlighted earlier, past work analyzing dual estimators growth bounds by upper and lower bounding growth of the population semi-dual functional, i.e.\,by studying the growth of $\semidualpq(\varphi) - \semidualpq(\varphi_0)$ (see~\eqref{eqn:dualstability}).
In contrast, when analyzing plugin estimators, it is
fruitful  to study the growth of the \emph{estimated} semi-dual functional: $\mathcal{S}_{\widehat{P},\widehat{Q}}(\varphi_0) - \mathcal{S}_{\widehat{P},\widehat{Q}}(\varphi_{\widehat{P},\widehat{Q}}).$ An important aspect of studying this latter quantity is that, due to the Bregman representation in Lemma~\ref{lem:bregman}, we are able to provide useful upper and lower bounds on the estimated semi-dual functional without requiring any regularity assumptions on the estimated Brenier potential~$\varphi_{\widehat{P},\widehat{Q}}.$ 

\vspace{.5cm} 

\noindent {\bf Stability Bounds for the OT Coupling: } 
Theorem~\ref{thm:stability}
can also be formulated 
as a stability bound over
the space of couplings. We first define
the Wasserstein distance
between the estimated
coupling $\hat\pi$ and $\pi_0 := (\mathrm{Id},T_0)_\# P$, namely:
\begin{align} \label{eq:W2_coupling} 
W^2_2(\hat\pi,\pi_0)
  = \inf_{\gamma\in\Pi(\hat\pi,\pi_0)}
 \int \|u - v\|_2^2 d\gamma(u,v).
\end{align}
Then, as a direct consquence of Corollary 3.9 of~\citet{li2021quantitative}, and our stability bound in~\eqref{eqn:mainstability} we obtain the following result:
\begin{corollary}
\label{cor:stability}
Under the conditions of Theorem~\ref{thm:stability}, 
there exists a constant $C > 0$
depending only on $\alpha$ and $\beta$
such that
\begin{align} 
\label{eq:coupling_stability}
W^2_2(\hat\pi,\pi_0) \leq 
C \big(   W^2_2(\widehat{P},P) + W^2_2(\widehat{Q},Q)  \big). 
\end{align} 
\end{corollary}
\noindent This Corollary elucidates a certain sharpness of  our stability bounds. Since
the cost function $\|u-v\|_2^2$ 
in~\eqref{eq:W2_coupling} depends
additively on its coordinates, 
it must be the case that the Wasserstein distance between couplings
is bounded from below by the
Wasserstein distance between
the corresponding marginal distributions, i.e.:
$$W^2_2(\hat\pi,\pi_0) \geq \max\big\{ W^2_2(\hat P,P), W^2_2(\hat Q,Q)\big\}.$$
It follows
that the bound in Corollary~\ref{cor:stability}
is sharp up to  constants.

\vspace{.5cm}

Finally, it is worth noting that in contrast to past work we obtain each of our stability bounds as direct consequences of a single unified result. We now briefly discuss a few highlights of the proof of Theorem~\ref{thm:stability} before providing some historical context for our work.

\vspace{.5cm}

\noindent {\bf Proof Highlights for Theorem~\ref{thm:stability}: } We give a full proof of this result in Appendix~\ref{app:stabilityproof}, but discuss some interesting aspects of the proof here.
Despite the fact that our goal is to prove a stability bound, we proceed by establishing a semi-dual growth bound. Concretely, by Lemma~\ref{lem:bregman} we have that,
\begin{align*} \mathcal{S}_{\widehat{P},\widehat{Q}}(\varphi_0) - \calS_{\widehat{P},\widehat{Q}}(\varphi_{\widehat{P},\widehat{Q}}) = \int \left[ \varphi^*_{0}(y) - 
\varphi^*_0(T_0(x)) - \inprod{\nabla \varphi_0^*(T_0(x))}{y - T_0(x)}\right] d \widehat{\pi}(x,y). 
\end{align*}
Our goal will be to upper and lower bound this quantity by the two sides of the inequality~\eqref{eqn:mainstability}, thus establishing the theorem.
Noting that $\varphi_0^*$ is $1/\beta$ strongly convex by~\ref{ass:smoothness} we have:
\begin{align*}
    \frac{1}{\beta} \mathbb{E}_{(X,Y) \sim \widehat{\pi}} \|Y - T_0(X)\|_2^2 \leq \mathcal{S}_{\widehat{P},\widehat{Q}}(\varphi_0) - \calS_{\widehat{P},\widehat{Q}}(\varphi_{\widehat{P},\widehat{Q}}),
\end{align*}
which immediately yields the desired lower bound. 

In contrast to the proof of the lower bound which only uses pointwise bounds (i.e. bounds which apply to the integrand pointwise), the proof of the upper bound is more involved using instead average-case optimality properties of the OT coupling~\eqref{eqn:maxcorr}. To illustrate the main ideas, let us consider the case when $\widehat{P}$ is taken to be the true distribution $P$, and focus on upper bounding the difference $\mathcal{S}_{P,\widehat{Q}}(\varphi_0) - \calS_{P,\widehat{Q}}(\varphi_{P,\widehat{Q}})$. 

Let $\widetilde{\pi}$ denote an optimal coupling of $P$ and $\widehat{Q}$, and 
let $(X,Y) \sim \tilde\pi$. 
Then, 
\begin{align*}
    \mathcal{S}_{P,\widehat{Q}}(\varphi_0) - \calS_{P,\widehat{Q}}(\varphi_{P,\widehat{Q}}) = \mathbb{E}\left[ \varphi_0^*(Y) - \varphi_0^*(T_0(X)) - \inprod{X}{Y - T_0(X)} \right],
\end{align*}
using the fact that $\nabla \varphi_0^*(T_0(X)) = X$ almost surely.
Now, given any random variable $Z$ with marginal distribution $\widehat{Q}$, we can equivalently write, 
\begin{align*}
    \mathcal{S}_{P,\widehat{Q}}(\varphi_0) - \calS_{P,\widehat{Q}}(\varphi_{P,\widehat{Q}}) = \mathbb{E} \left[ \varphi_0^*(Z) - \varphi_0^*(T_0(X)) - \inprod{X}{Z - T_0(X)} \right] - 
    \bbE \inprod{X}{Y - Z} .
\end{align*}
Recalling from~\eqref{eqn:maxcorr} that 
\begin{align*}
    \mathbb{E}_{(X,Y) \sim \widetilde{\pi}} \inprod{X}{Y} \geq \mathbb{E} \inprod{X}{Z},
\end{align*}
we obtain,
\begin{align*}
    \mathcal{S}_{P,\widehat{Q}}(\varphi_0) - \calS_{P,\widehat{Q}}(\varphi_{P,\widehat{Q}}) \leq  \mathbb{E} \left[ \varphi_0^*(Z) - \varphi_0^*(T_0(X)) - \inprod{X}{Z - T_0(X)}  \right].
\end{align*}
Using the $(1/\alpha)$-smoothness of $\varphi_0^*$ we obtain that,
\begin{align*}
    \mathcal{S}_{P,\widehat{Q}}(\varphi_0) - \calS_{P,\widehat{Q}}(\varphi_{P,\widehat{Q}}) \leq \frac{1}{\alpha} \mathbb{E} \|Z - T_0(X)\|_2^2.
\end{align*}
The above display holds
for any random variable $Z \sim \hat Q$, 
thus we are free to choose its joint distribution with $X$. 
Let $\pi_{Q,\widehat{Q}}$ denote an optimal coupling of $Q$ and $\widehat{Q}$, and pick $Z$ such that
$(T_0(X),Z) \sim \pi_{Q,\hat Q}$.
Then,
\begin{align*}
    \mathbb{E} \|Z -T_0(X)\|_2^2 = \mathbb{E}_{(U,V) \sim \pi_{Q,\widehat{Q}}} \|V-U\|_2^2 = W_2^2(Q, \widehat{Q}),
\end{align*}
which yields the bound:
\begin{align*}
\mathcal{S}_{P,\widehat{Q}}(\varphi_0) - \calS_{P,\widehat{Q}}(\varphi_{P,\widehat{Q}}) \leq \frac{1}{\alpha} W_2^2(Q,\widehat{Q}).
\end{align*}
The proof of the general upper bound, when $\widehat{P}$ is potentially different from $P$, requires a sharpening of the above argument (see Lemma~\ref{lem:onesample}) and we defer the details to Appendix~\ref{app:stabilityproof}.

\subsection{Related Work} 
With Theorem~\ref{thm:stability} in place we can now provide the appropriate context for our result. 
Stability bounds of the form we describe have a long history. For instance, a qualitative implication of the stability bound we derive in Corollary~\ref{cor:onesample} is that   the convergence of $\widehat{Q}$ to $Q$ in the $W_2$ metric implies the convergence of the OT map $\widehat{T}$ to $T_0$ in the $L^2(P)$ metric.
Qualitative results of this type are well-studied (for instance, see~\citet[Exercise 2.17]{villani2003}, 
\citet[Proposition 1.7.11]{panaretos2020},
and
\cite{segers2022}), and hold under significantly weaker conditions than the ones we impose. On the other hand, qualitative results do not typically yield sharp statistical rates on the risk in~\eqref{eqn:risk}.

To our knowledge, 
the earliest quantitative stability result
was derived by~\cite{gigli2011}
(who attributed it to Ambrosio)
 under the strong 
convexity 
assumption~\ref{ass:strong-convexity}
(see also Theorem~3.2 of \cite{ambrosio2019} and Theorem~3.5 of
\cite{li2021quantitative}).
Gigli showed that if  
$P\in \calP_{2,\mathrm{ac}}(\bbR^d)$, $\hat T$ is a   transport map from $P$ to $\hat Q$,
and $T_0=\nabla\varphi_0$ is the OT map from $P$ to $Q$, with the potential $\varphi_0$ satisfying condition~\ref{ass:strong-convexity}, then:
\begin{align*}
    \|T_0 - \hat T\|_{L^2(P)}^2 \lesssim \mathbb{E}_{X\sim P}\|\hat T(X) - X\|_2^2 - \mathbb{E}_{X \sim P} \|T_0(X) - X\|_2^2.
\end{align*}
This result suggests that the excess transport cost of a sub-optimal transport map grows in proportion to the $L^2(P)$ distance of the map from the OT map. This result is not directly useful for the statistical analysis of transport map estimates since it requires~$\hat T$ to be a valid transport map between $P$ and $\hat Q$, a condition that is not always satisfied by estimates of the map.
This deficiency inspired \citet{hutter2021} (see their Proposition~10) to develop the quadratic growth bound on the semi-dual functional $\semidualpq$, showing:
\begin{align*}
 \|\nabla \varphi - T_0\|_{L^2(P)}^2 \lesssim \semidualpq(\varphi) - \semidualpq(\varphi_0) \lesssim \|\nabla \varphi - T_0\|_{L^2(P)}^2,
\end{align*}
under the assumptions that $\varphi, \varphi_0$ are both smooth and strongly convex. \citet{vacher2021convex, muzellec2021} and \citet{makkuva2020} 
observed that the proof of \cite{hutter2021} also applies when $\varphi$ (alone) is smooth and strongly convex.
 
While the work of \citet{hutter2021} provides a semi-dual growth bound useful for the analysis of dual estimators, the works of \citet{ghosal2022,deb2021} and \citet{manole2024plugin} provide stability bounds useful for the analysis of plugin estimators. We have already discussed the results of \citet{manole2024plugin}, and highlighted how Theorem~\ref{thm:stability} provides useful improvements to their results. The stability bounds of~\citet{ghosal2022,deb2021} are quantitatively weaker than those in~\citet{manole2024plugin}, and replace the upper bound in~\eqref{eqn:mainstability}, by an empirical process term involving the estimated measures $\widehat{P}, \widehat{Q}$ and the corresponding Brenier potential $\varphi_{\widehat{P},\widehat{Q}}$. Analyzing this term to obtain risk bounds then necessitates strong assumptions to reason about the regularity of the Brenier potential $\varphi_{\widehat{P},\widehat{Q}}$ (akin to the torus assumption made by \citet{manole2024plugin}). 
A na\"{i}ve analysis of the empirical process term has a further deficiency which can result in sub-optimal rates of convergence for transport map estimates. In rough terms, the sharper upper bounds of Theorem~\ref{thm:stability} behave as the \emph{square} of an empirical process, yielding much faster rates of convergence.

Thus far, we have discussed
quantitative stability bounds
when one of conditions~\ref{ass:strong-convexity}
and~\ref{ass:smoothness} hold. 
In contrast, the works of \citet{berman2021convergence,merigot2020,delalande2023,gallouet2022strong}, and \citet{letrouit2024}  provide various stability bounds under variations of the target measure, which hold without smoothness assumptions on the optimal transport maps involved. That is, 
given a sufficiently regular
set $\Omega \subseteq \bbR^d$
and measure $P \in \calP_2(\Omega)$,
they derive bounds of the form:
\begin{equation}
\label{eq:merigot}
\|\hat T - T_0\|_{L^2(P)} \lesssim  W_2^\alpha(\hat Q,Q),
\end{equation}
for any measures $\hat Q,Q \in \calP(\Omega)$
satisfying appropriate tail conditions,
and with corresponding H\"older exponents $\alpha\in[0,1]$
which can be taken as high as $\alpha=1/6$.
Conversely, in this level
of generality, it
has been known since the work
of~\cite{gigli2011}
that the exponent $\alpha$ cannot be made greater than $1/2$, 
and it remains an open question to determine
whether this threshold can be achieved.
Our stability bounds show that 
the much more favorable H\"older exponent $\alpha=1$
is achievable when one of the optimal
transport maps is smooth. 

Beyond the statistical applications
that we have in mind, 
bounds of type~\eqref{eq:merigot} are highly sought after, as
the quantity $\|\hat T-T_0\|_{L^2(P)}$
is itself a metric between
$\hat Q$ and $Q$, which can be viewed as a proxy
for the Wasserstein distance. 
This quantity is sometimes
known as the linearized Wasserstein distance~\citep{wang2013linear},
and is widely-used in applications
due to its Hilbertian structure, and  its favorable computational properties~(e.g.\,\cite{cai2020linearized}). 
Corollary~\ref{cor:onesample} provides 
compelling motivation for the use of 
linearized Wasserstein distance, by showing
that it is in fact equivalent to the
original Wasserstein distance
under appropriate smoothness assumptions.

\section{Estimating OT Maps Between Smooth Distributions}
\label{sec:smooth}
As a first application, we will now show how Theorem~\ref{thm:stability}
can be used to characterize the risk of
plugin estimators of optimal transport maps
between smooth densities. 
Let 
$P$ and $Q$ denote two absolutely continuous
distributions over $\Omega$
with respective densities $p$ and $q$. 
To define our class of plugin estimators, 
let $\hat P_n$
denote an estimator of 
the distribution $P$ based
on the i.i.d. sample $X_1,\dots,X_n \sim P$.
We assume that $\hat P_n$ is  a {\it proper}
estimator, in the sense that it    
almost surely defines a probability measure 
in its own right. We further assume 
that $\hat P_n$ is absolutely continuous with respect
to the Lebesgue measure on $\Omega$, almost surely. 
Finally, we denote by $\hat Q_m$ any proper 
estimator of $Q$
based on the i.i.d. sample
$Y_1,\dots,Y_m \sim Q$. 
Due to the absolute continuity
of~$\hat P_n$, Theorem~\ref{thm:brenier} implies that there 
exists a unique optimal transport coupling between~$\hat P_n$ and~$\hat Q_m$ which is induced by an 
optimal transport map~$\hat T_{nm}$. 
This uniquely defined map~$\hat T_{nm}$ is a natural
plugin estimator of $T_0$, and 
Theorem~\ref{thm:stability} implies the
following bound on its $L^2(\hat P_n)$ risk:
$$\bbE \|\hat T_{nm} - T_0\|_{L^2(\hat P_n)}^2 
\lesssim \bbE \Big[W_2^2(\hat P_n,P) + 
 W_2^2(\hat Q_m,Q)\Big],$$
 provided that $T_0$ satisfies assumptions~\ref{ass:strong-convexity}--\ref{ass:smoothness}.
In order to characterize the $L^2(P)$ risk
of $\hat T_{nm}$ from here, one requires two ingredients:
\begin{enumerate}
\item {\bf Risk of Distribution Estimation in Wasserstein Distance.}
First, one needs to provide upper bounds on the risk of $\hat P_n$ and $\hat Q_n$ as estimates of $P$ and $Q$ in the Wasserstein distance. Bounds of this type are well-studied in the literature for a variety of estimators, including smooth density estimators~\citep{weed2019a,divol2022}, and 
the empirical measure~(e.g. \citep{boissard2014a,fournier2015, bobkov2019,weed2019,lei2020}, and references
therein).
\item {\bf Population and Empirical $L^2$ Norms.}
Second, one needs to relate the $L^2(\hat P_n)$ 
risk to the $L^2(P)$ risk. The simplest 
approach is to work with an estimator $\hat P_n$
for which the density ratio $dP / d\hat P_n$ is 
almost surely bounded over $\Omega$, 
in which case the $L^2(P)$ risk is trivially bounded
above by the $L^2(\hat P_n)$ risk: 
$$\bbE \|\hat T_{nm} - T_0\|_{L^2(P)}^2 
 \lesssim \bbE \|\hat T_{nm} - T_0\|_{L^2(\hat P_n)}^2.$$
When such a condition is not met, one
might instead try to bound the deviations
$$\Big| \|\hat T - T_0\|_{L^2(\hat P_n)}^2 - \|\hat T - T_0\|_{L^2(P)}^2\Big|
= \left| \int \|\hat T - T_0\|_2^2 d(\hat P_n-P)\right|,$$
uniformly over the set of allowable transport
maps $\hat T$. If $\hat P_n$ were replaced 
by the empirical measure of $X_1,\dots,X_n$, 
the above would simply be an empirical process indexed by the functions $\|\hat T-T_0\|_2^2$. Similar
quantities arise in the classical
theory of nonparametric least squares
regression, where 
localization and empirical process 
arguments are typically used to relate
the $L^2(P_n)$ risk of least squares
estimators to their $L^2(P)$ risk~\citep{vandegeer2000}.
Such arguments cannot directly
be translated to our setting since
$\hat P_n$ cannot be taken to be the
empirical measure, however
for some choices of this estimator, such as linear smoothers,
the above empirical process can
in principle still be controlled using 
known bounds on suprema of smoothed
empirical processes~\citep{radulovic2003,gine2009}. 
In Section~\ref{sec:logconcave}, 
we will illustrate a third
method for relating the $L^2(P)$
and $L^2(\hat P_n)$ norms, tailored to the nearest neighbor-based transport map estimator, which is 
based on characterizing the Voronoi cells generated by $X_1,\dots,X_n$.
\end{enumerate}

\noindent Let us now provide an example of an estimator
for which the above steps can be carried out. 
We adopt the same setting as~\citet[Section~4.4]{manole2024plugin},
who derive upper bounds on a smooth plugin
estimator {\it conditionally} on a boundary
regularity condition
(cf. condition~(C2) therein).
In contrast, we will derive an {\it unconditional}
version of their result in this section, 
which is enabled by the fact that our new stability bound in Theorem~\ref{thm:stability}
does not place any assumptions on the fitted
coupling.

We adopt the same notation and conditions as~\cite{manole2024plugin}. Concretely,  
assume that $\Omega$ is a {\it known}
compact, convex subset of $\bbR^d$ whose
boundary is $\calC^\infty$. 
Without loss of generality, 
let $\Omega$ have unit volume. 
Assume that the true
distributions $P$ and $Q$ have densities which lie in the following ball of H\"older continuous functions:
$$\calC^s(\Omega;M,\gamma)
 = \Big\{ f\in \calC^s(\Omega) : 
 \|f\|_{\calC^s(\Omega)} \leq M, 
 f \geq \gamma \text{ over } \Omega,
 D^k f = 0 \text{ on } \partial\Omega, k=1,\dots,\lfloor s\rfloor \Big\},
 $$
where $M,\gamma,s > 0$, and where
differentiation along the boundary
is to be understood in the 
weak sense~(cf. Appendix~J.1\,of~\cite{manole2024plugin}).
Beyond the smoothness
constraint, the class $\calC^s(\Omega;M,\gamma)$
requires the densities to be bounded away
from zero over the domain, and to have
vanishing derivatives at the boundary\footnote{The  condition of vanishing derivatives can be relaxed
to a Neumann boundary condition of sufficiently
high order, as in the definition of the function class $C_N^s(\Omega)$
of~\citet[Section 4.4]{manole2024plugin}.}.
It is well-known
that the Wasserstein density estimation
risk depends strongly on whether  or
not the densities are lower-bounded; 
we focus only on the
lower-bounded case, for which
the sharp minimax rate 
has been precisely characterized~\citep{bobkov2019,weed2019a,divol2022}. 
On the other hand, 
the boundary condition
allows us to construct density estimators
which are exactly supported on 
$\Omega$, which in turn 
is convenient for relating 
the Wasserstein density estimation risk to 
the simpler risk of estimating a density
under the norm of a 
particular Hilbert space~\citep{peyre2018}.
These are certainly not the most general conditions under which Wasserstein density
estimation can be studied, and improved results for Wasserstein density estimation will have direct
implications for the induced
optimal transport map estimators, due
to the modular nature of Theorem~\ref{thm:stability}.

Finally, in addition to the above conditions on the densities, we 
will assume that the $T_0$
satisfies the curvature conditions
described in Section~\ref{sec:curv}.
That is, we will assume that the 
true densities lies in the class 
\begin{align*} \calF^s(\Omega;M,\gamma,\alpha,\beta)= \big\{(p,q)\in \calD^2: p,q\in \calC^s(\Omega;M,\gamma),  &\text{ there exists }  
 \text{a Brenier potential } \varphi_0 \\
 &\text{ from } P \text{ to } Q  \text{ satisfying 
 \ref{ass:strong-convexity}--\ref{ass:smoothness}}\big\},
 \end{align*}
 where $\calD$ is the set of Lebesgue densities
 on $\bbR^d$.
By a well-known 
result of~\cite{caffarelli1996}, the condition 
$p,q\in \calC^s(\Omega;M,\gamma)$ 
implies the existence of a 
Brenier potential
from $P$ to $Q$
which satisfies
conditions~\hyperref[ass:strong-convexity]{\textbf{A1($\tilde\alpha$)}--\hyperref[ass:smoothness]{\textbf{A2($\tilde\beta$)}}}
for some $\tilde\alpha,\tilde\beta > 0$. To the best
of our knowledge, however,
there is no known
quantitative relation between
the resulting
parameters $\tilde\alpha,\tilde\beta$
and the original problem parameters $M,d,\gamma,s$,
thus we prefer to formulate the above assumption
on $\varphi_0$ directly. 
Similar considerations
are discussed by~\citet[Appendix E]{hutter2021}
and~\citet[Section 4.3]{manole2024plugin}.
 
The main result of this section is:
\begin{theorem}
\label{thm:smooth_densities}
For any $s > 0$, 
there exist proper and absolutely continuous  estimators 
$\hat P_n$ and $\hat Q_n$
such that
for any $M,\gamma,\alpha,\beta > 0$, there
exists a constant $C = C(M,\Omega,\gamma,s,\alpha,\beta) > 0$
for which the unique optimal
transport map $\hat T_n$ pushing
$\hat P_n$ forward onto $\hat Q_m$
satisfies
\begin{align} 
\label{eq:smooth_rate} 
\sup_{(p,q) \in \calF^s(\Omega;M,\gamma,\alpha,\beta)} 
R(\hat T_n,T_0) \leq 
C \epsilon_{n\wedge m},\quad \text{where } 
\epsilon_n:=\begin{cases} 
1/n, & d = 1, \\ 
\log n/n, & d = 2, \\
n^{-\frac{2(s+1)}{2s+d}}, & d \geq 3.
\end{cases}
\end{align} 
\end{theorem}
\begin{proof}
Assume without loss of generality
that $n=m$. Let $(p,q)\in \calF^s(\Omega;M,\gamma,\alpha,\beta)$.
We consider the estimators
$\hat P_n$ and $\hat Q_n$ defined
in Section~4.4 of~\cite{manole2024plugin}.
By Lemma~64 therein, there exists 
a constant $C_1 = C_1(M,\Omega,\gamma,s)$ and an
event $A_n$ of probability content at least
$1-C_1/n^2$
such that  the density $\hat p_n$ of $\hat P_n$
satisfies
$\hat p_n \geq 1 / C_1$
over $\Omega$.
Furthermore, since we have assumed that
$p\leq M$ over $\Omega$, 
we deduce that the density ratio
$dP/d\hat P_n$ is bounded over $\Omega$
by a constant depending only on $M,\gamma,d,s$, 
with probability at least $1-C_1/n^2$. 
Thus, 
\begin{align*}
\bbE \|\hat T_n - T_0\|_{L^2(P)}^2
 &\leq \bbE\Big[  \|\hat T_n - T_0\|_{L^2(P)}^2\,\Big|\, A_n\Big] + 
  \bbE\Big[  \|\hat T_n - T_0\|_{L^2(P)}^2 \,\Big|\, A_n^\cp\Big] \bbP(A_n^\cp) \\  
 &\lesssim \bbE \Big[\|\hat T_n - T_0\|_{L^2(\hat P_n)}^2\Big] 
  +  n^{-2},
\end{align*}
where we obtained the final term by noting
that $\|\hat T_n - T_0\|_{L^2(P)}^2$
is trivially bounded by a constant, due to the compactness of $\Omega$.  
We are now in a position to apply
Theorem~\ref{thm:stability}, which leads to:
\begin{align*}
\bbE \|\hat T_n - T_0\|_{L^2(P)}^2 
 &\lesssim  W_2^2(\hat P_n,P) +  W_2^2(\hat Q_n,Q) + n^{-2}.
\end{align*}
By Lemma~65 of~\cite{manole2024plugin}, 
the convergence rate of the density
estimators~$\hat P_n$ and~$\hat Q_n$ 
is bounded  above by
$\epsilon_n$ under
our conditions. The 
claim thus follows.
\end{proof}

Theorem~\ref{thm:smooth_densities}
is analogous
to Theorem~18 of~\cite{manole2024plugin}, with
the important difference
that it does not
rely on their assumption~(C2),
which they were unable to verify.
This assumption was
important for their work
since it implied that their estimator $\hat T_{nm}$ has an induced potential $\hat\varphi$ which
satisfies 
the regularity conditions~\ref{ass:strong-convexity}--\ref{ass:smoothness}
with high probability.
We are able to circumvent
this route since our
new stability bound
in Theorem~\ref{thm:stability}
places no smoothness assumptions on the fitted potential.

The convergence rate
appearing in Theorem~\ref{thm:smooth_densities}
matches known minimax lower bounds for estimating OT maps between H\"older-continuous densities~\citep{hutter2021}, up to a logarithmic factor when $d=2$. 
To our knowledge, this is the first
(unconditional) proof of minimax optimality of a two-sample plugin estimator for OT maps under H\"{o}lder regularity assumptions.

\section{An Improved Analysis of the Nearest-Neighbor Estimator}
\label{sec:nonsmooth}\label{sec:logconcave}

As a further illustration of  Theorem~\ref{thm:stability}, we consider the case where $\widehat{T}_{nm}$ is taken to be the 
empirical optimal transport map, extended to $\bbR^d$ via one-nearest-neighbor extrapolation. This  estimator has been studied in past work by~\cite{manole2024plugin}
(see also~\cite{pooladian2023minimax}),
however their analysis required
$P$ to be compactly-supported, 
and to 
admit a density bounded away
from zero over its  support. In what follows, 
we significantly weaken
these assumptions,
showing that the nearest-neighbor
estimator is minimax optimal
in estimating bi-Lipschitz
optimal transport maps, 
subject only to mild
moment constraints.

Concretely, given i.i.d. samples $X_1,\ldots,X_n \sim P$ and $Y_1,\ldots,Y_m \sim Q$, let
$\widehat{\pi}_{nm}$ be (the 
probability mass function of) an
optimal transport coupling 
between the empirical measures
$P_n$ and $Q_m$. That is,  
$$
\hat \pi_{nm} \in \argmin_{\pi}
\sum_{i=1}^n \sum_{j=1}^m \pi_{ij} \|X_i-Y_j\|^2,
$$
where the minimizer is over
all matrices $\pi\in \bbR^{n\times m}$
with nonnegative entries, whose
rows add up to $1/m$, and whose
columns add up to $1/n$.
Define the Voronoi partition induced by $X_1,\ldots,X_n$ as 
\begin{align*}
V_j = \{x \in \bbR^d: \norm{x-X_j} \leq \norm{x-X_i} , \  \forall i\neq j\}, \quad j=1, \dots, n.
\end{align*} 
Then, we define the one-nearest neighbor estimator by
\begin{align*}
\widehat T_{nm}^{\mathrm{1NN}}(x) = \sum_{i=1}^n \sum_{j=1}^m (n\widehat\pi_{ij}) I(x \in V_i) Y_j,\quad x \in \Omega.
\end{align*}
Corollary~\ref{cor:empirical} then immediately yields an empirical error bound for $\widehat {T}_{nm}^{\mathrm{1NN}}$: 
\begin{align*}
    \frac{1}{n} \sum_{i=1}^n 
    \|\widehat T_{nm}^{\mathrm{1NN}}(X_i) - T_0(X_i)\|_2^2 &= \frac{1}{n} \sum_{i=1}^n 
    \Big\|\sum_{j=1}^m n\widehat{\pi}_{ij} Y_j - T_0(X_i)\Big\|_2^2 \\
    &\leq \sum_{i=1}^n \sum_{j=1}^m \widehat{\pi}_{ij} \|Y_j - T_0(X_i)\|_2^2 \\
    &\lesssim W_2^2(Q_m,Q)  + W_2^2(P_n,P) \\
    &=: e_1. 
\end{align*}
In order to upper bound the $L^2(P)$ risk of $T_{nm}^{\mathrm{1NN}}$ we need to relate its in-sample error to the $L^2(P)$ error. We do so via a different recipe from that followed in Section~\ref{sec:smooth}. Concretely, suppose for any $X$ we define $\text{nn}(X)$ to be the nearest neighbor of $X$ in the sample $X_1,\ldots,X_n$. Then we define the expected distance from $X$ to its nearest neighbor as:
\begin{align*}
    e_2 := \mathbb{E}_{X \sim P} \left[\|X - \text{nn}(X) \|_2^2 | X_1,\ldots,X_n\right].
\end{align*}
We also define the maximum mass of any Voronoi cell by:
\begin{align*}
e_3 := \max_{i \in \{1,\ldots,n\}} P(V_i).
\end{align*}
With these quantities in place we show the following general result on the accuracy of the nearest neighbor map:
\begin{lemma}
\label{lem:nn}
Let $P,Q\in \calP_2(\bbR^d)$,
and suppose that assumptions~\ref{ass:brenier_regularity},  \ref{ass:strong-convexity} and~\ref{ass:smoothness} hold for some $\alpha, \beta > 0$. Then, 
there exists a constant
$C > 0$ depending only
on $\beta$ such that 
\begin{align*}
\|\widehat{T}_{nm}^{\mathrm{1NN}} - T_0\|^2_{L^2(P)} &\leq  
C\big(e_1 \times ne_3  + e_2\big).
 \end{align*}
\end{lemma}
\noindent We prove this result in Appendix~\ref{app:nn}. To get some sense of the strength of this result we highlight that each of the terms in the above bound can be controlled under relatively mild moment assumptions on the distributions $P$ and $Q$. Concretely, we suppose that for some $r > 0$, to be specified in the sequel, we have that:
\begin{align}
\label{eqn:rthmoment}
    \int \|x\|_2^r dP, \int \|x\|_2^r dQ \leq M_r^r < \infty.
\end{align}
\begin{lemma}
    \label{lem:moments}
Let $d > 4$, and let the measures 
$P, Q\in\calP(\bbR^d)$ 
satisfy~\eqref{eqn:rthmoment} for $r > 4$, and suppose that $P$ is non-atomic. Then, it holds that, for implicit constants depending only on $M_r,d,r$:
    \begin{align}
        \mathbb{E}[e_1^2] &\lesssim \left[m^{-4/d} + n^{-4/d}\right] \nonumber \\
        \mathbb{E}[e_2] &\lesssim n^{-2/d} \label{eqn:claimone}\\
        \mathbb{E}[e_3^{2}] &\lesssim \left[\frac{\log n}{n}\right]^2.\label{eqn:claimtwo}
    \end{align}
\end{lemma}

\noindent The first claim above follows directly from past work (e.g. \citet{fournier2015}). We prove the remaining claims in Appendix~\ref{app:momentslemma}. The restriction that $d > 4$ is not essential to our result, but different rates of convergence are obtained for $e_1$ and $e_2$ when $d \leq 4$. We do not aim to optimize the moment cutoff $r > 4$, but instead aim for concise and illustrative results and proofs. Under these mild conditions on the measures $P$ and $Q$ we obtain quantitative rates of convergence for a practical estimator of the transport map:
\begin{corollary}
\label{cor:nn}
Let $d > 4$, and let the measures
$P,Q\in \calP(\bbR^d)$ satisfy
condition~\eqref{eqn:rthmoment}. 
Suppose $P$ is non-atomic. 
Assume further that there exists
a Brenier potential $\varphi_0$
from $P$ to $Q$ satisfying conditions
\ref{ass:brenier_regularity}, \ref{ass:strong-convexity}, and \ref{ass:smoothness}. Then, 
there exists a constant $C > 0$
depending on $M_r,\alpha,\beta,d,r$
such that 
$$R(\hat T_{nm}^{\mathrm{1NN}}, T_0)
\leq C \big(n^{-2/d}+ m^{-2/d}\big) \log n.$$
\end{corollary}
\noindent Corollary~\ref{cor:nn} follows directly by combining Lemmas~\ref{lem:nn} and~\ref{lem:moments}, and an application of the Cauchy-Schwarz inequality. 
This result
shows that,
up to a logarithmic factor,  the nearest neighbor
estimator achieves the
 convergence rate $(n\wedge m)^{-2/d}$
for estimating a bi-Lipschitz optimal
transport map $T_0$,
and is therefore minimax optimal~\citep{hutter2021}. Our result merely
assumes that the measures admit $4+\epsilon$
moments for some $\epsilon > 0$;
this stands in 
contrast to past analyses of 
OT map estimation,
which
either
assume that $P$
is compactly-supported~(e.g. \cite{gunsilius2021,manole2024plugin}),
has exponential  
tails~\citep{deb2021,divol2022a}, 
or   yield
suboptimal convergence rates
when $P$ has finitely many
moments~\citep{ding2024}. 

\vspace{.2cm}

\noindent {\bf Transport between Log-Smooth and Strongly Log-Concave Distributions: } 
As an illustration of Corollary~\ref{cor:nn}, suppose we consider the case when $P$ and $Q$, are log-smooth and strongly log-concave. Concretely, $P$ and $Q$ are supported on $\mathbb{R}^d$ and admit Lebesgue densities of the form $P = \exp(-V)$ and $Q = \exp(-W)$, where $V$ and $W$ are twice-differentiable, and satisfy for all $x \in \mathbb{R}^d$: 
\begin{equation}\label{eq:log_concave} \alpha_V I \preceq \nabla^2 V(x) \preceq \beta_V I,\quad \text{and,}\quad \alpha_W I \preceq \nabla^2 W(x) \preceq \beta_W I.
\end{equation}

Under these assumptions, Caffarelli's contraction theorem~\citep{caffarelli2000}  implies that our regularity assumptions~\ref{ass:brenier_regularity}, \ref{ass:strong-convexity} and~\ref{ass:smoothness} hold for $\Omega = \mathbb{R}^d$, with $\alpha = \sqrt{\alpha_V/\beta_W}$ and $\beta = \sqrt{\beta_V/\alpha_W}$. Furthermore, the tails of a log-concave distribution are sub-exponential~\citep{ledoux2005concentration}, and these distributions satisfy the moment conditions of Lemma~\ref{lem:moments}. We obtain as a direct consequence of Corollary~\ref{cor:nn} the following result for the nearest-neighbor estimator:
\begin{corollary}
Let $P=\exp(-V)$ and $Q=\exp(-W)$ 
be log-concave measures with potentials satisfying~\eqref{eq:log_concave}.
Then, there exists a constant $C > 0$ depending only on $\alpha_V,\beta_V,\alpha_W,\beta_W$ such that
$$R(\hat T_{nm}^{\mathrm{1NN}}, T_0)
\leq C \big(n^{-2/d} + m^{-2/d}\big) \log n.$$
\end{corollary}

\noindent This result improves (by some logarithmic factors) the result of \citet{divol2022a}. 
More importantly, and in contrast to past work, our estimator of the transport map between two log-smooth and log-strongly concave distributions is simple and practical; in particular, 
it is free of tuning parameters.
 
\section{Discussion}
\label{sec:discussion}

Our results close an important gap in the literature on estimating smooth optimal transport maps
via the plugin principle. 
The main contribution
of our work was to 
derive a new two-sample stability bound in Theorem~\ref{thm:stability},
which does not rely on any regularity properties of the fitted optimal transport map.
Stability bounds of this form are useful as they reduce the study of plugin transport map estimators to that of distribution estimation under the Wasserstein loss. The latter is a well-studied problem with sharp results available for smoothness classes and under mild moment conditions. There are, however other arguments that one might be able to exploit to more directly upper bound the semi-dual difference, and we expect this to be a fruitful direction for future investigation.

We have centered our discussion around
two broad classes of OT map estimators:
plugin estimators, and dual estimators.
To our knowledge, 
these are the only methods which are known to be minimax optimal 
over typical classes of
smooth optimal transport maps. 
Quantitative convergence
rates have also been derived
for several other methods, 
based on
entropic  OT with vanishing regularization~\citep{pooladian2021,divol2024tight,eckstein2024convergence,mordant2024},
and sum-of-squares relaxations~\citep{vacher2021dimension,muzellec2021}. 
Although these methods 
are not known to be minimax optimal
in the same level of generality
as plugin or dual estimators, 
 they typically enjoy 
more favorable computational properties.

Our results relied
on smoothness and strong convexity
assumptions on the underlying Brenier potentials. These assumptions can be relaxed in various ways to study the estimation of nonsmooth optimal transport maps. This question has received some attention in the case when one of the measures $P$ and $Q$ is
discrete~\citep{pooladian2023minimax,sadhu2024approximation,sadhu2024stability}.

We have limited our attention
to the $L^2(P)$ risk of OT
map estimators, due to its close
connection with the optimality
property of the population OT map. 
It is nevertheless of interest
to understand the risk of OT map
estimators under other loss functions; 
for example, $L^\infty$ convergence
results are useful
in the study of Monge-Kantorovich
ranks and quantiles~\citep{hallin2021,ghosal2022}.
The recent works of~\cite{manole2023} and \cite{gonzalez2024} 
have established quantitative
stability bounds for the uniform
risk of plugin estimators, 
however they require strong
smoothness and boundary conditions on the
underlying distributions. It would be interesting to explore if the ideas from our work, which relaxes these types of regularity assumptions in $L^2(P)$ estimation, have any implications in more challenging settings.

\section*{Acknowledgements} 
The authors are grateful to Jonathan Niles-Weed and Larry Wasserman for numerous helpful conversations about stability bounds, and to Shayan Hundrieser for discussions related to this work. The work of SB was supported by the NSF
grant DMS-2310632. This work was done in part while SB was visiting the Simons Institute for the Theory of Computing.
The work of TM
is funded by a Norbert Wiener postdoctoral fellowship.

\bibliographystyle{abbrvnat}
\bibliography{stability}

\appendix

\section{The Semi-Dual Functional and Brenier Potentials}
\label{app:dual_props}

In this section, we recollect some well-known properties of the semi-dual functional, its optimizers, and some properties of convex conjugates that we use in our proofs.
Recall that for a given pair of distributions $P, Q \in \probspace$ we define a Brenier potential $\varphi_{P,Q}$ as any solution to the optimization problem:
\begin{align*}
        \varphi_{P,Q} \in \argmin_{\varphi \in L^1(P)} \int \varphi dP + \int \varphi^*dQ.
\end{align*}
If we denote by $\pi_{P,Q}$ an optimal coupling of $P$ and $Q$, then we have that
(see for instance Theorem 5.10 in \citet{villani2008}):
\begin{align*}
    \varphi_{P,Q}(X) + \varphi^*_{P,Q}(Y) = \inprod{X}{Y},
\end{align*}
for $\pi_{P,Q}$ almost every $(X,Y).$ Using this it is straightforward to see that,
\begin{align}
\label{eqn:semi-dual-opt}
    \mathcal{S}_{P,Q}(\varphi_{P,Q}) = \mathbb{E}_{(X,Y) \sim \pi_{P,Q}} \inprod{X}{Y}.
\end{align}
We also recall some well-known facts about convex conjugates (see~\cite{hiriart-urruty2004}), and derive some implications of the regularity assumption~\ref{ass:brenier_regularity}. 
For a convex function $\varphi$ defined on $\Omega$, we always have the Fenchel-Young inequality for any $x ,y\in \Omega$:
\begin{align*}
    \varphi(x) + \varphi^*(y) \geq \inprod{x}{y}.
\end{align*}
For a convex, differentiable function $\varphi$, the Fenchel-Young inequality holds with equality when $y = \nabla \varphi(x)$, i.e.:
\begin{align*}
    \varphi(x) + \varphi^*(\nabla \varphi(x)) = \inprod{x}{\nabla \varphi(x)}.
\end{align*}
If $\varphi$ is a strictly convex, differentiable function, then $\varphi^*$ is also convex and differentiable. In this case, we also have the identities:
\begin{align*}
\nabla\varphi^* \circ \nabla \varphi(x) = x, \quad \nabla\varphi\circ\nabla\varphi^*(y) = y.
\end{align*}

\section{Proof of Lemma~\ref{lem:bregman}}
\label{app:bregmanproof}

Recall that $\widetilde{\pi}$ denotes the optimal coupling between $\widetilde{P}$ and $\widetilde{Q}$.
We have that,
\begin{align*}
\mathcal{S}_{\widetilde{P}, \widetilde{Q}}(\varphi_0) - \mathcal{S}_{\widetilde{P}, \widetilde{Q}}(\widetilde \varphi) &= 
\mathbb{E}_{(X,Y) \sim \widetilde \pi} \left[  \varphi_0(X) + \varphi^*_0(Y) - \widetilde \varphi(X) - \widetilde \varphi^*(Y) \right].
\end{align*}
Now, we recall that:
\begin{align*}
\widetilde \varphi(X) + 
\widetilde \varphi^*(Y) 
= \inprod{X}{Y},~~~\widetilde \pi~\text{a.e.}
\end{align*}
By Assumption~\ref{ass:brenier_regularity} on $\varphi_0$, and using the Fenchel-Young (in)equality, we have that:
\begin{align*}
    \varphi_0(X) + \varphi_0^*(\nabla \varphi_0(X)) = \inprod{X}{\nabla\varphi_0(X)}.
\end{align*}
Putting these together, we obtain that:
\begin{align*}
\mathcal{S}_{\widetilde{P}, \widetilde{Q}}(\varphi_0) - \mathcal{S}_{\widetilde{P}, \widetilde{Q}}(\widetilde \varphi) &= \mathbb{E}_{(X,Y) \sim \widetilde \pi} \left[ \varphi^*_0(Y) - 
\varphi^*_0(T_0(X)) - \inprod{X}{Y - T_0(X)}\right],
\end{align*}
as desired.\qed 

\section{Proof of Theorem~\ref{thm:stability}}
\label{app:stabilityproof}

As discussed in the main text, our proof proceeds by upper and lower bounding the semi-dual difference $\mathcal{S}_{\widehat{P},\widehat{Q}}(\varphi_0) - \calS_{\widehat{P},\widehat{Q}}(\varphi_{\widehat{P},\widehat{Q}})$. The lower bound on this difference follows from the argument in the main text, and we provide a complete proof of the upper bound in this Appendix.

In the remainder of the proof, for any pair of distributions $P,Q$ we let $\pi_{P,Q}$ denote the optimal coupling between $P$ and $Q$. 
Then let us define the random variables
\begin{alignat*}{2}
X \sim P, \quad Y\sim Q,\quad 
U_1,U_2,U_3 \sim \hat P, \quad V_1,V_2,V_3 \sim \hat Q,    
\end{alignat*}
with the following joint distributions:
\begin{alignat*}{2}
(U_1,Y) &\sim \pi_{\hat P, Q}, \quad 
(X,V_1)&&\sim \pi_{P,\hat Q} \\
(X,U_2) &\sim \pi_{P,\hat  P},\quad 
(Y,V_2) &&\sim \pi_{Q,\hat Q} \\
(X,Y) &\sim \pi_{P,Q},\quad 
(U_3,V_3)&&\sim \pi_{\hat P,\hat Q}. 
\end{alignat*}
These joint distributions are summarized in the following figure:
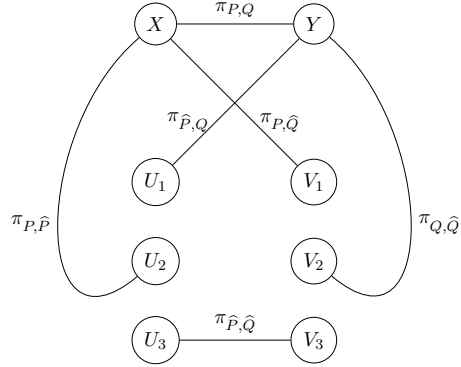
\begin{center}
\resizebox{0.4\textwidth}{!}{%
    \begin{tikzpicture}[
        node style/.style={draw, circle},  
        label style/.style={draw=none, fill=none} 
    ]
        \node[node style] (X) at (0, 6) {\(X\)};
        \node[node style] (Y) at (3, 6) {\(Y\)};
        \node[node style] (U1) at (0, 3) {\(U_1\)};
        \node[node style] (V1) at (3, 3) {\(V_1\)};
        \node[node style] (U2) at (0, 1.5) {\(U_2\)};
        \node[node style] (U3) at (0, 0) {\(U_3\)};
        \node[node style] (V2) at (3, 1.5) {\(V_2\)};
        \node[node style] (V3) at (3, 0) {\(V_3\)};
        
        \draw[-] (X) -- (Y) node[midway, above, label style] {\(\pi_{P,Q}\)};
        \draw[-] (U1) -- (Y) node[midway, above left, label style, xshift=-11pt, yshift=-20pt] {\(\pi_{\hat P,Q}\)};
        \draw[-] (X) -- (V1) node[midway, above right, label style, xshift=10pt, yshift=-20pt] {\(\pi_{P,\hat{Q}}\)};
        \draw[-] (U3) -- (V3) node[midway, above, label style] {\(\pi_{\hat{P},\hat{Q}}\)};
        
        \draw[-, out=220, in=220, looseness=1.5] (X) to node[midway, left, label style] {\(\pi_{P,\hat{P}}\)} (U2);
        \draw[-, out=320, in=320, looseness=1.5] (Y) to node[midway, right, label style] {\(\pi_{Q,\hat Q}\)} (V2);
    \end{tikzpicture} }
\end{center}

To begin we derive a technical result, giving one-sample stability bounds for the semi-dual functional:
\begin{lemma}
\label{lem:onesample}
Let $P, Q \in \probspace$, and assume~\ref{ass:brenier_regularity}, \ref{ass:strong-convexity} and~\ref{ass:smoothness} hold for some $\alpha, \beta > 0$. For any $\widehat{P}, \widehat{Q} \in \probspace$, 
\begin{align}
\mathcal{S}_{P,\widehat{Q}}(\varphi_0) - \mathcal{S}_{P,\widehat{Q}}(\varphi_{P,\widehat{Q}}) &\leq \frac{1}{2\alpha} W_2^2(\widehat{Q},Q) - \mathbb{E} \inprod{X}{V_1-V_2} \\
\mathcal{S}_{\widehat{P},Q}(\varphi_0) - \mathcal{S}_{\widehat{P},Q}(\varphi_{\widehat{P},Q}) &\leq \frac{\beta}{2} W_2^2(\widehat{P},P) - \mathbb{E} \inprod{Y}{U_1-U_2}.
\end{align}
\end{lemma}

\noindent With this result in place we return to our main proof. The following identity follows directly from the definition of the semi-dual functional:
\begin{align}
\label{eqn:f}
\semidualphatqhat(\varphi_0) - \semidualphatqhat(\varphi_{\widehat{P},\widehat{Q}}) = \semidualphatq(\varphi_0) + \semidualpqhat(\varphi_0) -
\semidualpq(\varphi_0) - \semidualphatqhat(\varphi_{\widehat{P},\widehat{Q}}).
\end{align}
Then we observe that by~\eqref{eqn:semi-dual-opt},
\begin{align*}
    \mathcal{S}_{\widehat{P},Q}(\varphi_{\widehat{P},Q}) &= \mathbb{E} \inprod{U_1}{Y} \\
    \mathcal{S}_{P,\widehat{Q}}(\varphi_{P,\widehat{Q}}) &= \mathbb{E} \inprod{X}{V_1} \\
    \mathcal{S}_{\widehat{P},\widehat{Q}}(\varphi_{\widehat{P},\widehat{Q}}) &= \mathbb{E} \inprod{U_3}{V_3} \\
    \mathcal{S}_{P,Q}(\varphi_0) &= \mathbb{E} \inprod{X}{Y}.
\end{align*}
From these facts, and~\eqref{eqn:f} we obtain,
\begin{align*}
\semidualphatqhat(\varphi_0) - \semidualphatqhat(\varphi_{\widehat{P},\widehat{Q}}) 
&=  \semidualphatq(\varphi_0) -  \semidualphatq(\varphi_{\widehat{P},Q}) + \semidualpqhat(\varphi_0) - \semidualpqhat(\varphi_{P,\widehat{Q}})  \\
&~~~~~~~~ +
\mathbb{E} \left[
  \inprod{U_1}{Y} + 
  \inprod{X}{V_1} -
  \inprod{X}{Y} - 
  \inprod{U_3}{V_3}
\right] \\
&\leq  \frac{1}{2\alpha}  W_2^2(\widehat{Q},Q)  + \frac{\beta}{2} W_2^2(\widehat{P},P) + \mathbb{E} \left[ \langle X, V_2-Y\rangle + \langle Y, U_2\rangle 
-\langle U_3,V_3\rangle \right],
\end{align*}
where the second inequality uses Lemma~\ref{lem:onesample}.
Now, by~\eqref{eqn:maxcorr}, we know that, $\mathbb{E} \inprod{U_3}{V_3} \geq \mathbb{E} \inprod{U_2}{V_2}$
and
so we obtain:
\begin{align*}
\semidualphatqhat(\varphi_0) - \semidualphatqhat(\varphi) &\leq \frac{1}{2\alpha}  W_2^2(\widehat{Q},Q)  + \frac{\beta}{2} W_2^2(\widehat{P},P) + \mathbb{E}\inprod{V_2-Y}{X-U_2}.
\end{align*}
Applying the Cauchy-Schwarz 
inequality to the final term and noting that $\mathbb{E} \|X-U_2\|_2^2 = W_2^2(\widehat{P},P)$ and $\mathbb{E} \|V_2-Y\|_2^2 = W_2^2(\widehat{Q},Q),$ we obtain the Theorem.\qed 

\subsection{Proof of Lemma~\ref{lem:onesample}}
The proofs of the two claims are similar, and we prove each in turn. For the first claim, notice that 
\begin{align*}
    \mathcal{S}_{P,\widehat{Q}}(\varphi_0) - \mathcal{S}_{P,\widehat Q}(\varphi_{P,\widehat Q}) &= \mathbb{E}  \left[ \varphi_0(X) + \varphi_0^*(V_1) - \varphi_{P,\widehat Q}(X) - \varphi_{ P,\widehat{Q}}^*(V_1) \right].
\end{align*}
Furthermore, it holds
almost surely that
\begin{align*}
    \varphi_{P,\widehat Q}(X) + \varphi_{P,\widehat Q}^*(V_1) 
     &= \inprod{X}{V_1},\quad 
    \varphi_0(X) + \varphi_0^*(Y)  = \inprod{X}{Y}.
\end{align*}
Now, since $V_2$ is equal in distribution to $V_1$, we have
\begin{align*}
    \mathcal{S}_{P,\widehat Q}(\varphi_0) - \mathcal{S}_{P,\widehat Q}(\varphi_{P,\widehat Q }) 
 &= \mathbb{E} \left[\varphi_0^*(V_1) - \varphi_0^*(Y) - \inprod{X}{ V_1-Y}\right] \\
 &=  \mathbb{E} \left[\varphi_0^*(V_2) - \varphi_0^*(Y) - \inprod{\nabla\varphi_0^*(Y)}{ V_2 - Y}\right] + \bbE \inprod{X}{V_2-V_1}.
\end{align*}
Using this fact together with the $1/\alpha$ smoothness of $\varphi_0$, we obtain the bound:
\begin{align*}
    \mathcal{S}_{P,\widehat Q}(\varphi_0) - \mathcal{S}_{P,\widehat Q}(\varphi_{P,\widehat Q}) &\leq
    \frac{1}{2\alpha } \mathbb{E} \|V_2 - Y\|_2^2 - \mathbb{E} \inprod{X}{V_1-V_2},
\end{align*}
which gives the first claim. 
To prove the second claim, 
note again that 
\begin{align*}
    \mathcal{S}_{\widehat{P},Q}(\varphi_0) - \mathcal{S}_{\widehat{P},Q}(\varphi_{\widehat{P},Q}) &= \mathbb{E}  \left[ \varphi_0(U_1) + \varphi_0^*(Y) - \varphi_{\hat P,Q}(U_1) - \varphi_{\hat P,Q}^*(Y) \right].
\end{align*}
Furthermore, it holds
almost surely that
\begin{align*}
    \varphi_{\hat P,Q}(U_1) + \varphi_{\hat P,Q}^*(Y) &= \inprod{U_1}{Y},\quad 
    \varphi_0(X) + \varphi_0^*(Y)  = \inprod{X}{Y}.
\end{align*}
Now, since $U_2$ is equal in distribution to $U_1$, we have
\begin{align*}
    \mathcal{S}_{\widehat{P},Q}(\varphi_0) - \mathcal{S}_{\widehat{P},Q}(\varphi_{\widehat{P},Q}) 
 &= \mathbb{E} \left[\varphi_0(U_1) - \varphi_0(X) - \inprod{Y}{U_1 - X}\right] \\
 &=  \mathbb{E} \left[\varphi_0(U_2) - \varphi_0(X) - \inprod{Y}{ U_2 - X}\right] - \bbE \inprod{Y}{U_1 - U_2}.
\end{align*}
Using this fact together with the smoothness of $\varphi_0$, we obtain the bound:
\begin{align*}
    \mathcal{S}_{\widehat{P},Q}(\varphi_0) - \mathcal{S}_{\widehat{P},Q}(\varphi_{\widehat{P},Q}) &\leq
    \frac{\beta}{2} \mathbb{E} \|U_2 - X\|_2^2 - \mathbb{E} \inprod{Y}{U_1-U_2},
\end{align*}
which gives the desired claim.\qed 

\section{Proof of Lemma~\ref{lem:nn}}
\label{app:nn}
For ease of notation we denote by $\widehat{T}$ the nearest neigbor OT map $\widehat{T}_{nm}^{\text{1NN}}$. We observe that we can decompose the $L^2(P)$ error of the nearest neighbor map as:
\begin{align*}
    \int \|\widehat{T} - T_0\|_2^2 dP &= \sum_{i=1}^n \int_{V_i} \|\widehat{T} - T_0\|_2^2 dP \\
    &\leq 2 \sum_{i=1}^n \int_{V_i} \|\widehat{T}(X_i) - T_0(X_i)\|_2^2 dP + 2 \sum_{i=1}^n \int_{V_i} \|T_0(x) - T_0(X_i)\|_2^2 dP(x) \\
    &\stackrel{\text{(i)}}{\leq} 2 \sum_{i=1}^n P(V_i) \|\widehat{T}(X_i) - T_0(X_i)\|_2^2 +2\beta^2 \sum_{i=1}^n \int_{V_i} \|x - X_i\|_2^2 dP(x) \\
    &\leq 2 \max_{i} P(V_i) \sum_{i=1}^n \|\widehat{T}(X_i) - T_0(X_i)\|_2^2 + 2 \beta^2 \mathbb{E}_{X \sim P} \left[\|X - \text{nn}(X)\|_2^2 | X_1,\ldots,X_n\right],
\end{align*}
as desired. The inequality (i) follows from~\ref{ass:smoothness}, noting that $T_0 = \nabla \varphi_0$, and that for a differentiable convex function, smoothness is equivalent to Lipschitzness of the gradient~\citep{hiriart-urruty2004}. 
\qed 

\section{Proof of Lemma~\ref{lem:moments}}
\label{app:momentslemma}
The first claim follows directly from past work, and we focus on the remaining claims. Similar results, albeit under much stronger boundedness assumptions, are classical in the analysis of $k$-NN regression (see, for instance, Chapter 6 of~\cite{gyorfi2006}).

\vspace{.1cm}

\noindent {\bf Proof of Claim~\eqref{eqn:claimone}: } Let $B_{x,r}$ denote the radius $r$ ball in $\mathbb{R}^d$ centered at $x$, and let $N(S,r)$ denote the $r$-covering number of the set $S\subseteq \bbR^d$ with respect to the Euclidean distance. We partition $\mathbb{R}^d$ into the sets:
\begin{align*}
S_0 = B_{0,1},\quad S_j = B_{0,2^j} \backslash B_{0,2^{j-1}},\quad j=1,2,\dots.
\end{align*}
Recalling that we denote by $\text{nn}(x)$ the nearest neighbor of $x$ in a sample $X_1,\ldots,X_n$, our goal is to bound:
\begin{align*}
\mathbb{E}[e_2] = \mathbb{E}_{X,X_1,\ldots,X_n \sim P} \|X - \nn(X)\|_2^2 = \int_{0}^{\infty}\bbP(\|X - \nn(X)\|_2 \geq \sqrt{t}) dt.
\end{align*}
For a fixed $x \in S_j$ we can write,
\begin{align*}
\bbP(\|x - \nn(x)\|_2 \geq \sqrt{t}) = (1 - P(B_{x,\sqrt{t}}))^n,
\end{align*}
so we obtain that,
\begin{align*}
\bbP(\|X - \nn(X)\|_2 \geq \sqrt{t}) &= \sum_{j} \int_{S_j}\bbP(\|x - \nn(x)\|_2 \geq \sqrt{t}) dP(x) \\
&= \sum_{j} \int_{S_j} (1 - P(B_{x,\sqrt{t}}))^n dP(x) \\
&\leq \sum_j \int_{S_j} \exp( - n P(B_{x,\sqrt{t}})) dP(x).
\end{align*}
Suppose we cover the set $S_j$ with radius $\sqrt{t}/2$ balls, and use the following implication
$$x \in B_{x^c,\sqrt{t}/2} \implies B_{x^c,\sqrt{t}/2} \subseteq B_{x, \sqrt{t}},$$
where $x^c$ is the center closest to $x$. Then,  
using the fact that 
$x \exp(-nx) \leq 1/n$, we get,
\begin{align*}
P(\|X - \nn(X)\|_2 \geq \sqrt{t}) &\leq \sum_j \min\{P(S_j), N(S_j, \sqrt{t}/2)/n\}.
\end{align*}
Now, using Markov's inequality
and the moment condition~\eqref{eqn:rthmoment},
we can upper bound the probability 
content of $S_j$ as follows:
\begin{align*}
    P(S_j) \leq P(\|x\|_2 \geq 2^{j-1}) \lesssim 2^{-r(j-1)}.
    \end{align*}
Furthermore,  noting that $N(S_j, \sqrt{t}) \lesssim (2^j/\sqrt{t})^d$,
we obtain
\begin{align*}
\mathbb{E} \|X - \nn(X)\|_2^2 &\lesssim \int_0^\infty \sum_j \min\Big\{2^{-r(j-1)}, \frac{(2^j/\sqrt{t})^d}{n}\Big\} dt \\
&\lesssim \sum_j \left[ 2^{-rj} \frac{2^{2j(d + r)/d}}{n^{2/d}} \right] \lesssim n^{-2/d},
\end{align*}
provided that $r \geq 2d/(d - 2)$.

\vspace{.1cm}

\noindent {\bf Proof of Claim~\eqref{eqn:claimtwo}: }
First, we note that by a standard uniform convergence argument (see, for instance, Lemma 7 in~\cite{chaudhuri2010} or~\cite{balakrishnan2013}), with probability at least $1 - \delta$,
any ball $B$ with $P(B) \gtrsim \log(n/\delta)/n$ must have at least one sample point in it. We condition on this event throughout the remainder of the proof.

We next use Lemma 6.2 from~\cite{gyorfi2006}, which shows the following. Suppose, for any $x$ and $a \geq 0$ we denote by $S_a(x)$ the set
\begin{align*}
S_a(x) := \{y: P(B_{y,\|x - y\|_2}) \leq a\}.
\end{align*}
Then, since $P$ is non-atomic, 
it holds for every $x \in \mathbb{R}^d$ that, 
\begin{align}
\label{eqn:gyorfi}
    P(S_a(x))
\lesssim a.
\end{align}
For any sample point $x \in \{X_1,\ldots,X_n\}$ let us denote by $V_x$ the Voronoi cell containing $x$. 
If $y \in V_x$, then $B_{y,\|x-y\|_2}$ must not contain any sample point other than $x$. This in turn means that $y \in S_a(x)$ for $a \asymp \log(n/\delta)/n$. Consequently, using the fact that $P$ is non-atomic, we obtain from~\eqref{eqn:gyorfi} that:
\begin{align*}
    P(V_x) \lesssim \frac{\log(n/\delta)}{n}.
\end{align*}
Thus, letting $Z := \max_i P(V_i)$, we have shown that with probability at least $1 - \delta$, 
\begin{align*}
    Z \lesssim \frac{\log(n/\delta)}{n}.
\end{align*}
From this we obtain that 
\begin{align*}
    \mathbb{E}[Z^2] \lesssim \delta + \left[\frac{\log(n/\delta)}{n}\right]^2 \lesssim \left[\frac{\log n}{n}\right]^2,
\end{align*}
by choosing $\delta$ appropriately.
\qed 
\end{document}